\documentclass[reqno]{amsart}
\usepackage{amsmath, amsfonts, amsthm, amssymb, parskip, setspace, textcomp,
	fullpage, bbm}
\newtheorem{theorem}{Theorem}[section]

\newtheorem{corollary}[theorem]{Corollary}
\newtheorem{proposition}{Proposition}

\theoremstyle{definition}

\theoremstyle{remark}

\numberwithin{equation}{section}

\newcommand{\spt}[1]{\mbox{\normalfont spt}\Parans{#1}}

\newcommand{\Mspt}[2]{\mbox{\normalfont M2spt}_{#1}\Parans{#2}}

\newcommand{\Parans}[1]{\left(#1\right)}
\newcommand{\CBrackets}[1]{\left\{#1\right\}}
\newcommand{\SBrackets}[1]{\left[#1\right]}

\newcommand{\PieceTwo}[4]
{
	\left\{
   	\begin{array}{ll}
      	#1 & #3 \\
       	#2 & #4
     	\end{array}
	\right.
}
\newcommand{\aqprod}[3]{\Parans{#1;#2}_{#3}}
\newcommand{\Jac}[2]{\left(\frac{#1}{#2}\right)}

\newcommand{\Bin}[2]
{
	\left(
   	\begin{array}{c}
      	#1\\
       	#2
     	\end{array}
	\right)
}

\newcommand{\Floor}[1]{\lfloor #1 \rfloor}

\newcommand{\Delz}{\delta_z}
\newcommand{\Delq}{\delta_q}

\newcommand{\CommonSpace}[1]{
	\frac{\aqprod{-q}{q^2}{\infty}}{\aqprod{q^2}{q^2}{\infty}}
	\cdot\mathcal{W}_{#1}(\Gamma_0(4))
}

\author{CHRIS JENNINGS-SHAFFER}
\address{Department of Mathematics, University of Florida\\
Gainesville, Florida 32611, USA
\endgraf cjenningsshaffer@ufl.edu}



\title{Rank and Crank Moments for Partitions without Repeated Odd Parts}

\allowdisplaybreaks
\begin{document}

\allowdisplaybreaks

\begin{abstract}
Using quasimodular forms with respect to $\Gamma_0(4)$ we find exact relations between the
M2-rank for partitions without repeated odd parts and three residual cranks.
From these identities we are able to deduce various congruences mod $3$ and $5$
between the rank and crank moments. In turn, these congruences give
congruences for $\Mspt{}{n}$, the number of occurrences of smallest parts in
the partitions of $n$ with smallest part even and without repeated odd parts,
and for the higher order analog $\Mspt{2}{n}$.
\end{abstract}

\maketitle

\section{Introduction and Statement of Results}

\allowdisplaybreaks

In \cite{AG} Atkin and Garvan found exact relations between the moments of the 
rank and crank generating function of partitions. This came out of a partial 
differential equation satisfied by the rank and crank. Garvan in \cite{Garvan4} 
used these equations to deduce various congruences for Andrews' spt function
and again used the equations in \cite{Garvan5} to prove congruences for higher
order analogs of the spt function. 
In \cite{BLO2}, this same technique was used by Bringmann, Lovejoy, and Osburn 
to prove congruences for spt functions related to overpartitions. We use this idea
to prove congruences for the M2spt function as well as a higher order analog.
First we review the necessary definitions.

We recall a partition of a positive integer $n$ is a non-increasing sequence of positive
integers that sum to $n$. For example, the partitions of $4$ are
$4$, $3+1$, $2+2$, $2+1+1$, and $1+1+1+1$. We let $p(n)$ denote the number
of partitions of $n$, so $p(4)=5$. In \cite{Andrews} Andrews introduced the smallest
parts function, $\spt{n}$, a weighted count of the partitions of $n$. The idea is the 
count each partition by the number of times the smallest part occurs. We then
have $\spt{4} = 10$.

For other types of partitions it is also natural to consider a corresponding
smallest parts function. The restrictions we consider are partitions where the
odd parts do not repeat and the smallest part is even. For example, such 
partitions of $11$ are $9+2$, $7+4$, $7+2+2$, $6+3+2$, $5+4+2$, $5+2+2+2$,
$4+3+2+2$, and $3+2+2+2+2$. We let $\Mspt{}{n}$ denote the number of smallest parts
in these partitions of $n$, so $\Mspt{}{11}=15$.

Two statistics often associated to partitions are the rank and crank of a 
partition. One point of interest is that these statistics explain certain
linear congruences satisfied by $p(n)$. 
The rank of a partition is the largest part minus the number of parts. We let $N(m,n)$
denote the number of partitions of $n$ with rank $m$ and set
\begin{align*}
	R(z,q) 
	&= 
	\sum_{n=0}^\infty\sum_{m=-\infty}^\infty N(m,n)z^mq^n
.
\end{align*}
The crank of a partition is the largest part if there are no ones and otherwise
is the number of parts larger than the number of ones minus the number of ones.
After suitably altering the interpretations for $n=0$ and $n=1$, one has that 
\begin{align*}
	C(z,q)
	&=
	\sum_{n=0}^\infty\sum_{m=-\infty}^\infty M(m,n)z^mq^n
	=
	\frac{\aqprod{q}{q}{\infty}}{\aqprod{zq,z^{-1}q}{q}{\infty}}
.
\end{align*}
We then have the rank and crank moments given by
\begin{align*}
	&N_k(n) = \sum_{m=-\infty}^\infty m^k N(m,n)
	,
	&M_k(n) = \sum_{m=-\infty}^\infty m^k M(m,n)
.
\end{align*}
We will use two differential operators:
\begin{align*}
	\Delq &= q\cdot \frac{\partial}{\partial q}
	,
	&\Delz = z\cdot \frac{\partial}{\partial z}
.
\end{align*}
The generating functions for the rank and crank moments are given by
\begin{align*}
	R_k &= R_k(q) = \sum_{n=0}^\infty N_k(n) q^n
	= \Delz^k R(z,q)  |_{z=1}
	,\\
	C_k &= C_k(q) = \sum_{n=0}^\infty M_k(n) q^n
	= \Delz^k C(z,q)  |_{z=1}
.
\end{align*}
We note for odd $k$ that these moments are zero.

We recall the weight $k$ Eisenstein series is given by
\begin{align*}
	E_{k}(q) &= 1 - \frac{2k}{B_{k}}\sum_{n=1}^\infty \sigma_{k-1}(n)q^n
,
\end{align*}
where $B_k$ are the Bernoulli numbers and $\sigma_{k-1}$
is a sum of divisors function
\begin{align*}
	\sigma_{k-1}(n) &= \sum_{d\mid n} d^{k-1}
.
\end{align*}
We note that 
\begin{align*}
	E_{k}(q) 
	&= 
	1 - \frac{2k}{B_{k}}\sum_{n=1}^\infty \frac{n^{k-1}q^n}{1-q^n}
.
\end{align*}
For even $k>2$, $E_{k}(q)$ is the expansion at $\infty$ of a modular form of 
weight $k$ for the full modular 
group $\Gamma_0(1)$.

We let $M_{k}(\Gamma_0(N),\chi)$ denote the finite dimensional vector space
of modular forms of weight $k$ with respect to the congruence subgroup 
$\Gamma_0(N)$ of $\Gamma_0(1)$ with character $\chi$. The various facts we use
about modular forms can be found in \cite{Ono1}. Generally we will work with modular forms in
terms of their $q$-series expansions, rather than as a function of $\tau$ in the
upper half plane with $q=e^{2\pi i\tau}$. However at times it will be convenient
to also work with $\tau$, equations in which both $q$ and $\tau$ appear should 
be interpreted in terms of $q$-series.

We let $\mathcal{W}_{n}(\Gamma_0(N))$ denote the vector space of quasimodular forms
with respect to $\Gamma_0(N)$ of weight at most $2n$ and with zero constant term in the
$q$-expansion at infinity. Elements of this vector space can be viewed as polynomials
in $E_2(q)$ and holomorphic modular forms of non-negative weight on 
$\Gamma_0(N)$ where each monomial is of weight at most $2n$ (viewing $E_2(q)$ as
having weight $2$). 

We will work with vector spaces of the form 
$A\cdot\mathcal{W}_{n}(\Gamma_0(N))$, where $A$ is some fixed product. We note
the dimension of $A\cdot\mathcal{W}_{n}(\Gamma_0(N))$ is the same as
$\mathcal{W}_{n}(\Gamma_0(N))$

In \cite{AG} Atkin and Garvan proved each $\Delq^j C_{2k}$ is an element of
$\frac{1}{\aqprod{q}{q}{\infty}}\cdot\mathcal{W}_{j+k}(\Gamma_0(1))$. Although
each individual $\Delq^j R_{2k}$ is not necessarily in any of the spaces
$\frac{1}{\aqprod{q}{q}{\infty}}\cdot\mathcal{W}_{m}(\Gamma_0(1))$,
these spaces do contain certain linear combinations of the $\Delq^j R_{2k}$. 
Bringmann, Lovejoy, and Osburn proved similar facts for 
overpartitions in \cite{BLO2}. Here we prove the corresponding theorem for
partitions without repeated odd parts.

The $M_2$-rank of a partition $\pi$ without repeated odd
parts is given by
\begin{align}
	M_2\mbox{-rank} = \left\lceil\frac{l(\pi)}{2} \right\rceil - \#(\pi),
\end{align}
where $l(\pi)$ is the largest part of $\pi$ and $\#(\pi)$ is the number of
parts of $\pi$. The $M_2$-rank was introduced by Berkovich and Garvan
in \cite{BG2} and was further studied by Lovejoy and Osburn in \cite{LO2}. 
We let $N2(m,n)$ denote the number of partitions of $n$ with
distinct odd parts 
and $M_2$-rank $m$ and let
\begin{align*}
	R2(z,q) 
	&= 
	\sum_{n=0}^\infty\sum_{m=-\infty}^\infty N2(m,n)z^mq^n
.
\end{align*}

We use a residual crank from \cite{GarvanJennings}. For a partition $\pi$ 
without repeated odd parts, we consider the crank of the partition 
$\frac{\pi_e}{2}$ obtained by omitting the odd parts and halving each 
even part of $\pi$. 
Letting $M2(m,n)$ denote the number of partitions of $n$ without repeated 
odd parts and this residual crank $m$, upon suitably
altering this definition for such partitions whose only even part is a single
$2$, we find the generating function is given by
\begin{align*}
	C2(z,q)
	&=
	\sum_{n=0}^\infty\sum_{m=-\infty}^\infty M2(m,n)z^mq^n
	= 
		\frac{\aqprod{-q}{q^2}{\infty}\aqprod{q^2}{q^2}{\infty}}
			{\aqprod{zq^2}{q^2}{\infty}\aqprod{z^{-1}q^2}{q^2}{\infty}}
	.
\end{align*}

It turns out we will need two more residual crank like functions.
We let
\begin{align}
	C1(z,q)
	&=
	\sum_{n=0}^\infty\sum_{m=-\infty}^\infty M1(m,n)z^mq^n
	= 
		\frac{\aqprod{q^2}{q^4}{\infty}\aqprod{q}{q}{\infty}}
			{\aqprod{zq}{q}{\infty}\aqprod{z^{-1}q}{q}{\infty}}
	,\\
	C4(z,q)
	&=
	\sum_{n=0}^\infty\sum_{m=-\infty}^\infty M4(m,n)z^mq^n
	= 
		\frac{\aqprod{q^4}{q^4}{\infty}}
			{\aqprod{q}{q^2}{\infty}\aqprod{zq^4}{q^4}{\infty}\aqprod{z^{-1}q^4}{q^4}{\infty}}
.
\end{align}
While $M4(m,n)$ can be interpreted as a residual crank, 
such an interpretation for $M1(m,n)$ is not immediately clear.

We use the various moments
\begin{align*}
	&M1_k(n) = \sum_{m=-\infty}^\infty m^k M1(m,n)
	,
	&M2_k(n) = \sum_{m=-\infty}^\infty m^k M2(m,n)
	,\\
	&M4_k(n) = \sum_{m=-\infty}^\infty m^k M4(m,n)
	,
	&N2_k(n) = \sum_{m=-\infty}^\infty m^k N2(m,n) 
\end{align*}
and their generating functions
\begin{align*}
	C1_k &= C1_k(q) = \sum_{n=0}^\infty M1_k(n) q^n
	= \Delz^k C1(z,q)  |_{z=1}
	,\\ 
	C2_k &= C2_k(q) = \sum_{n=0}^\infty M2_k(n) q^n
	= \Delz^k C2(z,q)  |_{z=1}
	,\\
	C4_k &= C4_k(q) = \sum_{n=0}^\infty M4_k(n) q^n
	= \Delz^k C4(z,q)  |_{z=1}
	,\\ 
	R2_k &= R2_k(q) = \sum_{n=0}^\infty N2_k(n) q^n
	= \Delz^k R2(z,q)  |_{z=1}
.
\end{align*}
We note for odd $k$ that these moments are zero. To reduce the length of our
calculations, we will often write the various functions without their 
dependence on $z$ and $q$.

The purpose of this paper is to establish the following theorem and
use it to prove various congruences.
\begin{theorem}\label{TheoremMain}
For $N\ge 1$ the following are elements of 
$\frac{\aqprod{-q}{q^2}{\infty}}{\aqprod{q^2}{q^2}{\infty}}\cdot\mathcal{W}_{N}(\Gamma_0(4))$,
\begin{enumerate}
\item[(i)]
	the functions $\Delq^m(C1_{2j})$
	for $m\ge 0$, $1\le j\le N$, $j+m\le N$, 
\item[(ii)]
	the functions $\Delq^m(C2_{2j})$
for $m\ge 0$, $1\le j\le N$, $j+m\le N$, 
\item[(ii)]
	the functions $\Delq^m(C4_{2j})$
for $m\ge 0$, $1\le j\le N$, $j+m\le N$, 
\item[(iv)]
the function
\begin{align*}
	&(a^2 -3a + 2)R2_a
	+\sum_{k=1}^{a/2-1} \Bin{a}{2k} (2^{2k+1} - 4) \Delq R2_{a-2k}
	\\&
	+\sum_{k=1}^{a/2-1} 
	\Parans{2\Bin{a}{2k} -2^{2k+1}\Bin{a}{2k+1} + (2^{2k+2}-2)\Bin{a}{2k+2} } R2_{a-2k}
,
\end{align*}
where $a=2N$.	
\end{enumerate}
\end{theorem}

In Section 2 we prove Theorem \ref{TheoremMain}, in Section 3
we obtain exact relations between the rank and crank moments, and in
Section 4 we use these relations to prove various congruences for 
$\Mspt{}{n}$ and a higher order analog $\Mspt{2}{n}$.


\section{Proof of Theorem \ref{TheoremMain}}
The proof follows in the same fashion as a similar fact for the rank and crank
of a partition \cite{AG} and both the Dyson rank and M2-rank of overpartitions
along with the relevant residual cranks \cite{BLO2}.

Thinking in terms of $q$-expansions, $\Delq$ maps 
$\mathcal{W}_{n}(\Gamma_0(N))$ to $\mathcal{W}_{n+1}(\Gamma_0(N))$.
This follows from Proposition 2.11 of \cite{Ono1}, which gives that $\Delq$ takes
modular forms to quasimodular forms, and by Ramanujan's identity
that $\Delq E_2(q) = (E_2(q)^2-E_4(q))/12$. 
We note replacing $q$ by $q^m$ maps
$M_{n}(\Gamma_0(N))$ to $M_{n}(\Gamma_0(mN))$.
Since $E_2(q)-mE_2(q^m)\in M_{2}(\Gamma_0(m))$ for any positive 
integer $m$, we see that replacing $q$ by $q^m$ also maps
$\mathcal{W}_{n}(\Gamma_0(N))$ to $\mathcal{W}_{n}(\Gamma_0(mN))$.

By Atkin and Garvan \cite{AG} we know for $j\ge 1$ that $\Delz^{2j}C(z,q)|_{z=1}$ is an 
element of 
$\frac{1}{\aqprod{q}{q}{\infty}}\cdot\mathcal{W}_{j}(\Gamma_0(1))$. 
We note for $j$ odd that
$\Delz^jC(z,q)|_{z=1}=0$. Thus for $j\ge 1$ we have
$\Delz^{2j}C(z,q^2)|_{z=1}$ in 
$\frac{1}{\aqprod{q^2}{q^2}{\infty}}\mathcal{W}_{j}(\Gamma_0(2))$
and for odd $j$ we have $\Delz^jC(z,q^2)|_{z=1}=0$. 
Also we note for $j=0$ we have
$C(z,q^2)|_{z=1} = \frac{1}{\aqprod{q^2}{q^2}{\infty}}$.

But then $C1_{2j}$, $C2_{2j}$, and $C4_{2j}$, for positive $j$, are elements of
$\frac{\aqprod{-q}{q^2}{\infty}}{\aqprod{q^2}{q^2}{\infty}}\mathcal{W}_{j}(\Gamma_0(4))$.
We will show that 
$\Delq\Parans{\frac{\aqprod{-q}{q^2}{\infty}}{\aqprod{q^2}{q^2}{\infty}}}$
is an element of 
$\frac{\aqprod{-q}{q^2}{\infty}}{\aqprod{q^2}{q^2}{\infty}}\cdot\mathcal{W}_{1}(\Gamma_0(4))$.
Thus if $f$ is an element of 
$\frac{\aqprod{-q}{q^2}{\infty}}{\aqprod{q^2}{q^2}{\infty}}\cdot\mathcal{W}_{k}(\Gamma_0(4))$,
then $\Delq(f)$ is an element of
$\frac{\aqprod{-q}{q^2}{\infty}}{\aqprod{q^2}{q^2}{\infty}}\cdot\mathcal{W}_{k+1}(\Gamma_0(4))$,
by induction and the product rule.
This would prove parts $(i)$, $(ii)$, and $(iii)$.

\begin{proposition}
$\Delq \Parans{ \frac{\aqprod{-q}{q^2}{\infty}}{\aqprod{q^2}{q^2}{\infty}} }$ is an element of
$\frac{\aqprod{-q}{q^2}{\infty}}{\aqprod{q^2}{q^2}{\infty}}\cdot\mathcal{W}_{1}(\Gamma_0(4))$.
\end{proposition}
\begin{proof}
We find
\begin{align*}
	\Delq \frac{\aqprod{-q}{q^2}{\infty}}{\aqprod{q^2}{q^2}{\infty}}
	&=
	\frac{\aqprod{-q}{q^2}{\infty}}{\aqprod{q^2}{q^2}{\infty}}
	\Parans{
		\sum_{n=0}^\infty \frac{(2n+1)q^{2n+1}}{1+q^{2n+1}}
		+\sum_{n=1}^\infty \frac{2nq^{2n}}{1-q^{2n}}
	}
	\\
	&=
	\frac{\aqprod{-q}{q^2}{\infty}}{\aqprod{q^2}{q^2}{\infty}}
	\Parans{
		\sum_{n=1}^\infty \frac{nq^{n}}{1-q^{n}}
		-\sum_{n=1}^\infty \frac{2nq^{2n}}{1-q^{2n}}
		+\sum_{n=1}^\infty \frac{4nq^{4n}}{1-q^{4n}}
	}
	\\
	&=
	\frac{-\aqprod{-q}{q^2}{\infty}}{24\aqprod{q^2}{q^2}{\infty}}
	\Parans{
		E_2(q)
		-2E_2(q^2)
		+4E_2(q^4)
		-3
	}
.
\end{align*}
Since $E_2(q)-2E_2(q^2)+4E_2(q^4)-3$ is a quasimodular form and in the 
$q$-expansion there is no constant term, we have 
$\Delq\Parans{ \frac{\aqprod{-q}{q^2}{\infty}}{\aqprod{q^2}{q^2}{\infty}}}$ 
is an element
of 
$\frac{\aqprod{-q}{q^2}{\infty}}{\aqprod{q^2}{q^2}{\infty}}\cdot\mathcal{W}_{1}(\Gamma_0(4))$.
\end{proof}

For part $(iv)$, we start with a partial differential equation proved by 
Bringmann, Lovejoy, and Osburn in \cite{BLO1}:
\begin{align}\label{EqThePde}
	&2z\frac{\aqprod{q^2}{q^2}{\infty}^2}{\aqprod{-q}{q^2}{\infty}}
	\SBrackets{C(z,q^2)}^3\aqprod{-qz,-q/z}{q^2}{\infty}	
	&=
	\Parans{2(1-z)^2\Delq + (1+z)(1-z)\Delz + 2z + (1-z)^2\Delz^2}	R2(z,q)
.
\end{align}

We note that for $n\ge 1$
\begin{align*}
	\Delz^n z^m &= m^n z^m,
\end{align*}
in particular
\begin{align*}
	&\Delz^n z = z
	,&
	\Delz^n z^2 = 2^n z^2
.
\end{align*}

Also by Leibniz's rule
\begin{align*}
	\Delz^n (f\cdot g)
	&=
	\sum_{k=0}^n \Bin{n}{k} (\Delz^k f)(\Delz^{n-k} g)
,
\end{align*}
thus
\begin{align*}
	\Delz^n (z \cdot f)|_{z=1}
	&=
	\sum_{k=0}^n \Bin{n}{k} \Delz^{n-k} f
	,&
	\Delz^n (z^2 \cdot f)|_{z=1}
	=
	\sum_{k=0}^n \Bin{n}{k} 2^k \Delz^{n-k} f
.
\end{align*}

The right hand side of (\ref{EqThePde}) is
\begin{align*}
	&(2z^2-4z+2)\Delq R2(z,q)
	+(-z^2+1)\Delz R2(z,q)
	+2z R2(z,q)
	+(z^2-2z+1)\Delz^2 R2(z,q)
,
\end{align*}
so applying $\Delz^a$, for positive even $a$, and setting $z=1$ gives
\begin{align*}
	&\sum_{k=0}^a \Bin{a}{k} (2^{k+1} - 4) \Delz^{a-k} \Delq R2(z,q) |_{z=1}
	+2\Delz^a \Delq R2(z,q)	|_{z=1}
   \\&
	-\sum_{k=0}^a \Bin{a}{k} 2^k \Delz^{a-k+1} R2(z,q) |_{z=1}
	+\Delz^{a+1} R2(z,q) |_{z=1}
	+2\sum_{k=0}^a \Bin{a}{k} \Delz^{a-k} R2(z,q) |_{z=1}
	\\&
	+\sum_{k=0}^a \Bin{a}{k} (2^k-2) \Delz^{a-k+2} R2(z,q) |_{z=1}
	+\Delz^{a+2} R2(z,q) |_{z=1}
	\\
	=&
		\sum_{k=1}^a \Bin{a}{k} (2^{k+1} - 4) \Delq R2_{a-k}
		-\sum_{k=0}^{a-1} \Bin{a}{k+1} 2^{k+1} R2_{a-k}
		+2\sum_{k=0}^a \Bin{a}{k} R2_{a-k}
		\\&
		+\sum_{k=-1}^{a-2} \Bin{a}{k+2} (2^{k+2}-2) R2_{a-k}
.
\end{align*}	

Noting $R2_m = 0$ for $m$ odd, the above is
\begin{align*}
	&
	\sum_{k=1}^{a/2-1} \Bin{a}{2k} (2^{2k+1} - 4) \Delq R2_{a-2k}
	\\&
	+\sum_{k=1}^{a/2-1} 
	\Parans{2\Bin{a}{2k} -2^{2k+1}\Bin{a}{2k+1} + (2^{2k+2}-2)\Bin{a}{2k+2} } R2_{a-2k}
	\\&
	+ (2^{a+1} -4)\Delq R2_0
	+ (a^2 -3a + 2)R2_a
	+ 2R2_0
.
\end{align*}

Applying $\Delz^a$ to the left hand side of (\ref{EqThePde}) we get
\begin{align*}
	&2\frac{\aqprod{q^2}{q^2}{\infty}^2}{\aqprod{-q}{q^2}{\infty}}
	\sum_{k=0}^a
	\Bin{a}{k}
	\Delz^k(zC(z,q^2)^3)  |_{z=1}
	\Delz^{a-k}(\aqprod{-zq}{-q/z}{\infty})  |_{z=1}
	\\
	=&
	2\frac{\aqprod{q^2}{q^2}{\infty}^2}{\aqprod{-q}{q^2}{\infty}}
	\sum_{k=0}^{a-1}
	\sum_{j=0}^k
	\Bin{a}{k}\Bin{k}{j}
	\Delz^j(C(z,q^2)^3) |_{z=1}
	\Delz^{a-k}(\aqprod{-zq}{-q/z}{\infty})|_{z=1}
		\\&+
		2\aqprod{q^2}{q^2}{\infty}^2\aqprod{-q}{q^2}{\infty}
		\sum_{j=1}^a
		\Bin{a}{j}
		\Delz^j(C(z,q^2)^3) |_{z=1}
		+
		2\frac{\aqprod{-q}{q^2}{\infty}}{\aqprod{q^2}{q^2}{\infty}}
.
\end{align*}

Noting $M2_0 = R2_0 = \frac{\aqprod{-q}{q^2}{\infty}}{\aqprod{q^2}{q^2}{\infty}}$, we
have now
\begin{align}\label{EqReducedPde}
	&-(2^{a+1} -4)\Delq R2_0
		+2\frac{\aqprod{q^2}{q^2}{\infty}^2}{\aqprod{-q}{q^2}{\infty}}
		\sum_{k=0}^{a-1}
		\sum_{j=0}^k
		\Bin{a}{k}\Bin{k}{j}
		\Delz^j(C(z,q^2)^3) |_{z=1}
		\Delz^{a-k}(\aqprod{-zq}{-q/z}{\infty})|_{z=1}
		\nonumber\\&+
		2\aqprod{q^2}{q^2}{\infty}^2\aqprod{-q}{q^2}{\infty}
		\sum_{j=1}^a
		\Bin{a}{j}
		\Delz^j(C(z,q^2)^3) |_{z=1}
	\nonumber\\
	=&
	(a^2 -3a + 2)R2_a
	+\sum_{k=1}^{a/2-1} \Bin{a}{2k} (2^{2k+1} - 4) \Delq R2_{a-2k}
		\nonumber\\&
		+\sum_{k=1}^{a/2-1} 
		\Parans{2\Bin{a}{2k} -2^{2k+1}\Bin{a}{2k+1} + (2^{2k+2}-2)\Bin{a}{2k+2} } R2_{a-2k}
.
\end{align}

To prove the theorem, we show the left hand side
of (\ref{EqReducedPde}) is in 
$\frac{\aqprod{-q}{q^2}{\infty}}{\aqprod{q^2}{q^2}{\infty}}\cdot\mathcal{W}_{a/2}(\Gamma_0(4))$. 
We do this by working term by term.

By applying the Leibniz rule and examining terms, we find that
for $j$ odd 
$\Delz^j (C(z,q^2)^3) |_{z=1} = 0$
and for $j$ even and positive
$\Delz^j (C(z,q^2)^3) |_{z=1}$
is in 
$\frac{1}{\aqprod{q^2}{q^2}{\infty}^3}\cdot\mathcal{W}_{j/2}(\Gamma_0(2))$. 
Thus
\begin{align*}
	2\aqprod{q^2}{q^2}{\infty}^2\aqprod{-q}{q^2}{\infty}
	\sum_{j=1}^a
	\Bin{a}{j}
	\Delz^j(C(z,q^2)^3) |_{z=1}
\end{align*}
is and element of
$\frac{\aqprod{-q}{q^2}{\infty}}{\aqprod{q^2}{q^2}{\infty}}\cdot\mathcal{W}_{a/2}(\Gamma_0(2))$.

For $\aqprod{-qz,-q/z}{q^2}{\infty}$, we start by noting
\begin{align*}
	\Delz \aqprod{-qz,-q/z}{q^2}{\infty}
	&=
	\SBrackets{
		\sum_{n\ge 0}\frac{zq^{2n+1}}{1+zq^{2n+1}}
		-\sum_{n\ge 0}\frac{z^{-1}q^{2n+1}}{1+z^{-1}q^{2n+1}}
	}
	\aqprod{-qz,-q/z}{q^2}{\infty}
	\\
	&=\SBrackets{
		\sum_{n\ge 0}\sum_{m\ge 1} 
		(-1)^{m}(z^{-m}-z^m) q^{m(2n+1)}
	}
	\aqprod{-qz,-q/z}{q^2}{\infty}
.
\end{align*}

We set
\begin{align*}
	F(z,q) &= 
	\sum_{n\ge 0}\sum_{m\ge 1} 
	(-1)^{m}(z^{-m}-z^m) q^{m(2n+1)}
.
\end{align*}

\begin{proposition}
For $\ell$ even, $\Delz^{\ell}F(z,q) |_{z=1} = 0$. For $\ell$ odd,
$\Delz^{\ell}F(z,q) |_{z=1}$ is an element of 
$\mathcal{W}_{(\ell+1)/2}(\Gamma_0(4))$. 
\end{proposition}
\begin{proof}
We see
\begin{align*}
	\Delz^{\ell}	F(z,q) |_{z=1}
	&=
	\sum_{n\ge 0}\sum_{m\ge 1} 
		(-1)^{m}((-1)^\ell m^\ell -m^\ell ) q^{m(2n+1)}
	\\
	&=
	\PieceTwo{0}{-2\sum_{n\ge 0}\sum_{m\ge 1} (-1)^m m^\ell q^{m(2n+1)}}
	{\ell \mbox{ even}}{\ell \mbox{ odd}}
.
\end{align*}
Next we use
\begin{align*}
	-2\sum_{n\ge 0}\sum_{m\ge 1} (-1)^{m}m^\ell  q^{m(2n+1)}
	&=
	-4\sum_{m\ge 1} \frac{(2m)^\ell q^{2m}}{1-q^{2m}}
	+2\sum_{m\ge 1} \frac{m^\ell q^m}{1-q^{2m}}
.
\end{align*}
We note
\begin{align*}
	2\sum_{m\ge 1} \frac{m^\ell q^m}{1-q^{2m}}
	&=
	2\sum_{m\ge 1} \frac{m^\ell q^m}{1-q^m}
	-2\sum_{m\ge 1} \frac{m^\ell q^{2m}}{1-q^{2m}}
	=
	-\frac{B_{\ell+1}}{\ell+1}\Parans{E_{\ell+1}(q)-E_{\ell+1}(q^2)}
.
\end{align*}

Thus
\begin{align*}
	&-2\sum_{n\ge 0}\sum_{m\ge 1} (-1)^{m}m^\ell  q^{m(2n+1)}
	=
	-\frac{B_{\ell+1}}{\ell+1}\Parans{
		E_{\ell+1}(q)-E_{\ell+1}(q^2)-2^{\ell+1}E_{\ell+1}(q^2)
		+2^{\ell+1}E_{\ell+1}(q^4)
	}
.
\end{align*}
For $\ell>1$ odd this is a modular form of weight $\ell+1$ with respect to $\Gamma_0(4)$,
for $\ell=1$ this is a quasimodular form of weight $2$ with respect to $\Gamma_0(4)$,
noting there is no constant term we see
$\Delz^{\ell}F(z,q) |_{z=1}$ is in
$\mathcal{W}_{(\ell+1)/2}(\Gamma_0(4))$.

\end{proof}

By inducting on $j$, we find for $j\ge 1$ that
$\Delz^j \aqprod{-qz,-q/z}{q^2}{\infty}|_{z=1} \in 
\aqprod{-q}{q^2}{\infty}^2\mathcal{W}_{\Floor{j/2}}(\Gamma_0(4))$.

Working term by term we then find

\begin{align*}
	&2\frac{\aqprod{q^2}{q^2}{\infty}^2}{\aqprod{-q}{q^2}{\infty}}
	\sum_{k=0}^{a-1}
	\sum_{j=0}^k
	\Bin{a}{k}\Bin{k}{j}
	\Delz^j(C(z,q^2)^3) |_{z=1}
	\Delz^{a-k}(\aqprod{-zq}{-q/z}{\infty})|_{z=1}
\end{align*}
is an element of 
$\frac{\aqprod{-q}{q^2}{\infty}}{\aqprod{q^2}{q^2}{\infty}}
\mathcal{W}_{a/2}(\Gamma_0(4))$.

The only remaining term is $-(2^{a+1} -4)\Delq R2_0$. But
$R2_0 = \frac{\aqprod{-q}{q^2}{\infty}}{\aqprod{q^2}{q^2}{\infty}}$
and we have already verified 
$\Delq\Parans{\frac{\aqprod{-q}{q^2}{\infty}}{\aqprod{q^2}{q^2}{\infty}}}$ is an 
element of
$\frac{\aqprod{-q}{q^2}{\infty}}{\aqprod{q^2}{q^2}{\infty}}
\mathcal{W}_{1}(\Gamma_0(4))$. This accounts for all the terms on the left hand
side of (\ref{EqReducedPde}), so the theorem is proved.

\section{Exact Relations}

We note since every quasimodular form with respect to $\Gamma_0(N)$ can be written 
uniquely as a polynomial in $E_2$ with coefficients that are modular forms 
with respect to
$\Gamma_0(N)$, this is by Proposition 1 of \cite{MZ}, each element of 
$\mathcal{W}_N(\Gamma_0(4))$ can be written uniquely in the form
\begin{align*}
	\sum_{n=0}^N E_2(q)^n f_n(q)
	,\hspace{30pt}\mbox{where }
	f_n &\in \sum_{j=0}^{2N-2n} M_{j}(\Gamma_0(4))
. 
\end{align*}
However, not all elements of this form are in $\mathcal{W}_N(\Gamma_0(4))$, since we
require the elements of $\mathcal{W}_N(\Gamma_0(4))$ to have zero constant term.
In terms of computing the dimension of $\mathcal{W}_N(\Gamma_0(4))$, this simply
means we must subtract $1$.
That is,
\begin{align*}
	Dim(\mathcal{W}_N(\Gamma_0(4))) 
	&=  
	\Parans{\sum_{n=0}^{N}\sum_{j=0}^{2N-2n} Dim(M_{j}(\Gamma_0(4)))} - 1
	=
	N+\sum_{n=0}^{N}\sum_{j=1}^{N-n} Dim(M_{2j}(\Gamma_0(4)))
.
\end{align*}

In particular $\CBrackets{Dim(\mathcal{W}_N(\Gamma_0(4)))}_{N=1}^\infty = 
\CBrackets{3,9,19,34,55,\dots}$. We recall that
\\ $Dim\Parans{\CommonSpace{N}} = Dim(\mathcal{W}_N(\Gamma_0(4)))$.

For $N=2$ and $3$ we will find the function in part $(iv)$ of Theorem 
\ref{TheoremMain} to be in the span of the functions from parts $(i)$, $(ii)$,
and $(iii)$. We list these exact relations. The $N=1$ case is of no interest
since the function in part $(iv)$ is zero.

\begin{corollary}
For $n\ge 0$,
\begin{align}
\label{RTwo4ExactRelation}
	6N2_4(n) + (24n-6)N2_2(n)
	&=
	\frac{516-1356n}{469} M1_2(n) 
	+\frac{120}{67} M1_4(n)  
	+\frac{960-360n}{469} M2_2(n)
	-\frac{4896}{469} M2_4(n)
	\nonumber\\&\quad 
	-\frac{2976+5424n}{469} M4_2(n)  
	+\frac{1920}{67} M4_4(n) 
.
\end{align}
\end{corollary}
\begin{proof}
With $N=2$ we have $Dim(\mathcal{W}_2(\Gamma_0(4)))=9$ and find, using Maple, 
that the $9$ functions from parts $(i)$, $(ii)$, and $(iii)$ of Theorem 
\ref{TheoremMain} are linearly independent.

Expressing the function from part $(iv)$ in terms of this basis is
\begin{align*}
	6R2_4 + 24\Delq R2_2 - 6R2_2
	&=
	\frac{516}{469} C1_2 
	-\frac{1356}{469} \Delq C1_2
	+\frac{120}{67} C1_4  
	+\frac{960}{469} C2_2 
	-\frac{360}{469} \Delq C2_2
	\\&\quad
	-\frac{4896}{469} C2_4
	-\frac{2976}{469} C4_2 
	-\frac{5424}{469} \Delq C4_2 
	+\frac{1920}{67} C4_4 
.
\end{align*}
Expressing this identity in terms of the coefficients of the series is
(\ref{RTwo4ExactRelation}). 
\end{proof}

We need one more element of $\CommonSpace{3}$.
By Theorems 1.64 and 1.65 of \cite{Ono1}, we know $\eta(2\tau)^{12}$ to be an
element of $M_6(\Gamma_0(4))$ and vanishes at the cusp infinity. Thus
\begin{align*}
	F(q) &= q\aqprod{-q}{q^2}{\infty}\aqprod{q^2}{q^2}{\infty}^{11}
\end{align*}
is an element of
$\CommonSpace{3}$.

\begin{corollary}\label{CorollaryRTwo6ExactRelation}
For $n\ge 0$,
\begin{align}
	\label{RTwo6ExactRelation}
	&20N2_6(n) + (80+60n)N2_4(n) + (420n-100)N2_2(n)
	\nonumber\\
	&=
	\frac{21624800 - 61258080n + 5880120n^2}{1119503} M1_2(n) 
	+\frac{5256200-584400n}{159929}M1_4(n) 
	+\frac{320}{341}M1_6(n)
	\nonumber\\&\quad
	+\frac{35188800 - 11366640n - 1945200n^2}{1119503}M2_2(n) 
	-\frac{187116960 + 3942720n}{1119503}M2_4(n) 
	-\frac{7680}{341}M2_6(n)
	\nonumber\\&\quad	
	+\Parans{\frac{-114563200+23520480n^2}{1119503}-\frac{20044320n}{101773}}M4_2(n) 
	+\Parans{\frac{2422400}{5159}	- \frac{9350400n}{159929}}M4_4(n)
	\nonumber\\&\quad
	+\frac{20480}{341}M4_6(n)	
.
\end{align}
\end{corollary}
\begin{proof}
With $N=3$ we have $Dim(\mathcal{W}_2(\Gamma_0(4)))=19$ and find, using Maple, 
that the $18$ functions from parts $(i)$, $(ii)$, and $(iii)$ of Theorem 
\ref{TheoremMain} along with $F$ are linearly independent. Expressing the function
in part $(iv)$ in terms of this basis gives
\begin{align*}
	&20R2_6 + 60\Delq R2_4 + 420\Delq R2_2 + 80R2_4 - 100R2_2
	\\
	&=
	\frac{21624800}{1119503}C1_2
	-\frac{61258080}{1119503}\Delq C1_2
	+\frac{5880120}{1119503}\Delq^2 C1_2 
	+\frac{5256200}{159929}C1_4 
	-\frac{584400}{159929}\Delq C1_4
	+\frac{320}{341}C1_6
	\\&\quad
	+\frac{35188800}{1119503}C2_2 
	-\frac{11366640}{1119503}\Delq C2_2 
	-\frac{1945200}{1119503}\Delq^2 C2_2 
	-\frac{187116960}{1119503}C2_4 
	-\frac{3942720}{1119503}\Delq C2_4 
	-\frac{7680}{341}C2_6
	\\&\quad
	-\frac{114563200}{1119503}C4_2 
	-\frac{20044320}{101773}\Delq C4_2 
	+\frac{23520480}{1119503}\Delq^2 C4_2
	+\frac{2422400}{5159}C4_4 
	-\frac{9350400}{159929}\Delq C4_4
	+\frac{20480}{341}C4_6	
.
\end{align*}
Expressing this identity in terms of the coefficients of the series is
(\ref{RTwo6ExactRelation}).
\end{proof}
It is somewhat surprising that while the functions from $(i)$, $(ii)$, and $(iii)$ do 
not give a basis for 
$\CommonSpace{3}$, the function in part $(iv)$ is indeed in their
span.
For higher values of $N$, one can take the functions from parts  $(i)$, 
$(ii)$, and $(iii)$ and complete them to a basis for the appropriate space by
adding in known modular forms with respect to $\Gamma_0(4)$, 
however in these cases it no longer appears that the function from part $(iv)$
is in the span of just the functions from parts $(i)$, $(ii)$, and 
$(iii)$.

To derive all of our congruences we will need one more relation, here we need 
one more rank function. An overpartition of $n$ is a partition of $n$ in which
the first occurrence of a part may, or may not, be overlined. The 
overpartitions of $4$ are $4$, $\overline{4}$, $3+1$, $3+\overline{1}$, 
$\overline{3}+1$, $\overline{3}+\overline{1}$, $2+2$, $\overline{2}+2$, 
$1+1+1+1$, and $\overline{1}+1+1+1$. The Dyson rank of an 
overpartition is the largest part minus the number of parts. We let 
$\overline{N}(m,n)$ denote the number of overpartitions of $n$ with Dyson rank $m$. We set
\begin{align*}
	\overline{R}(z,q) 
	&= 
	\sum_{n=0}^\infty\sum_{m=-\infty}^\infty \overline{N}(m,n)z^mq^n
	,\\
	\overline{N}_k(n) &= \sum_{m=-\infty}^\infty m^k \overline{N}(m,n)
	,\\
	\overline{R}_k &= \overline{R}_k(q) = \sum_{n=0}^\infty \overline{N}_k(n) q^n
	= \Delz^k\ \overline{R}(z,q)  |_{z=1}
.
\end{align*}

\begin{corollary}
\begin{align}\label{EquationExtraRelation}
	F1 - \frac{159}{64}F2
	&=
	\frac{948341197409}{633638698}C1_2 
	-\frac{318249663559}{1267277396}\Delq C1_2
	+\frac{18906057102}{316819349}\Delq^2 C1_2 
	-\frac{221063911175}{181039628}C1_4
	\nonumber\\&\quad 
	-\frac{1124944110}{45259907}\Delq C1_4
	-\frac{11439407}{193006}C1_6
	-\frac{8682641651833}{5069109584}C2_2 
	-\frac{724498277229}{633638698}\Delq C2_2
	\nonumber\\&\quad
	+\frac{14799375252}{316819349}\Delq^2 C2_2
	+\frac{2398983090355}{1267277396}C2_4 
	-\frac{59855835000}{316819349}\Delq C2_4 
	+\frac{12021538}{96503}C2_6
	\nonumber\\&\quad
	+\frac{13424561341}{633638698}C4_2 
	-\frac{15708001159}{28801759}\Delq C4_2 
	+\frac{37605906528}{316819349}\Delq^2 C4_2
	+\frac{1078788930}{1459997}C4_4 
	\nonumber\\&\quad 
	-\frac{7136728080}{45259907}\Delq C4_4
	+\frac{4510496}{96503}C4_6
,
\end{align}
where
\begin{align*}
	F1 &= F1(q) =  \aqprod{q^2}{q^4}{\infty}\Parans{10R_6 + 90\Delq R_4 + 630\Delq R_2 + 40R_4 - 50R_2}
	,\\
	F2 &= F2(q) =  \frac{\aqprod{q}{q^2}{\infty}}{\aqprod{-q^2}{q^2}{\infty}}
		\Parans{20\overline{R}_6 + 120\Delq\overline{R}_4 + 1920\Delq\overline{R}_2 
			+ 275\overline{R}_4 + 215\overline{R}_2}
.
\end{align*}
\end{corollary}
\begin{proof}
By the first line of the proof of Theorem 5.1 of \cite{AG} we know
\begin{align*}
	&10R_6 + 90\Delq R_4 + 630\Delq R_2 + 40R_4 - 50R_2
\end{align*} 
to be an element of 
$\frac{1}{\aqprod{q}{q}{\infty}}\cdot\mathcal{W}_3(\Gamma_0(1))$ and
by Theorem 1.1 of \cite{BLO2} we know
\begin{align*}
	&20\overline{R}_6 + 120\Delq\overline{R}_4 + 1920\Delq\overline{R}_2 
	+ 275\overline{R}_4 + 215\overline{R}_2
\end{align*} 
to be an element of 
$\frac{\aqprod{-q}{q}{\infty}}{\aqprod{q}{q}{\infty}}\cdot\mathcal{W}_3(\Gamma_0(2))$.
Thus $F1(q)$ and $F2(q)$ are elements of 
$\frac{\aqprod{-q}{q^2}{\infty}}{\aqprod{q^2}{q^2}{\infty}}\cdot\mathcal{W}_3(\Gamma_0(2))$.

Expressing $F1-\frac{159}{64}F2$ in terms of the basis from Corollary 
\ref{CorollaryRTwo6ExactRelation} is (\ref{EquationExtraRelation}).
\end{proof}
The reason for choosing this specific combination of $F1$ and $F2$ is that 
to express $F1$ or $F2$ in terms of the chosen basis, the function $F$ is
required, however this combination eliminates the need for $F$.

\section{Congruences For $\Mspt{1}{n}$ and $\Mspt{2}{n}$.}

We recall $\Mspt{}{n}$ is the total number of occurrences of the smallest parts in
the partitions of $n$ with smallest part even and without repeated odd parts. 
In \cite{JenningsShaffer} the author introduced a generalization of 
$\Mspt{}{n}$ given by
\begin{align*}
	\Mspt{k}{n} &= \mu 2_{2k}(n) - \eta 2_{2k}(n)
,
\end{align*}
where
\begin{align*}
	\eta 2_k(n) 
	&= 
	\sum_{m\in\mathbb{Z}} \Bin{m+\Floor{\frac{k-1}{2}}}{k} N2(m,n),
	&
	\mu 2_k(n) 
	= 
	\sum_{m\in\mathbb{Z}} \Bin{m+\Floor{\frac{k-1}{2}}}{k} M2(m,n)
.
\end{align*}
One finds that $\Mspt{1}{n} = \Mspt{}{n}$. These higher order generalizations 
have a combinatorial interpretation as a weighted count of partitions 
based on the frequency of various parts. In particular, letting S2 
denote the set of partitions with with smallest part even and without repeated
odd parts and letting $f_j(\pi)$ denote the frequency of the $j$-th even part of such
a partition $\pi$, we have
\begin{align*}
	\Mspt{2}{n}
	&=
	\sum_{\substack{\pi\in\mbox{S2}\\ |\pi|=n }}
		\Parans{
			\Bin{f_1(\pi)+1}{3}
			+f_1(\pi)\sum_{m\ge 2} \Bin{f_m(\pi) + 1}{2}
		}
.
\end{align*}

It turns out
\begin{align*}
	&\eta 2_{2k}(n) 
	= 
	\frac{1}{(2k)!}
	\sum_{m\in\mathbb{Z}} g_k(m) N2(m,n),
	&\mu 2_{2k}(n) 
	=
	\frac{1}{(2k)!} 
	\sum_{m\in\mathbb{Z}} g_k(m) M2(m,n)
,
\end{align*}
where $g_k(x) = \prod_{j=0}^{k-1}(x^2-j^2)$. One then computes that
\begin{align}
	\label{SptEquation1}
	\Mspt{}{n} &= \Mspt{1}{n} = \frac{1}{2}\Parans{M2_2(n)-N2_2(n)}
	,\\\label{SptEquation2}
	\Mspt{2}{n} &=
	\frac{1}{24}\Parans{ M2_4(n) - M2_2(n) - N2_4(n) + N2_2(n) 	}
.
\end{align}

We prove the following congruences:
\begin{theorem}\label{TheoremCongruences}
For $n\ge 0$,
\begin{align}
	\label{Congruence3n1}
	\Mspt{}{3n+1}\equiv 0 \pmod{3}
	,\\
	\label{Congruence5n1A}
	\Mspt{}{5n+1}\equiv 0 \pmod{5}
	,\\
	\label{Congruence5n3A}
	\Mspt{}{5n+3}\equiv 0 \pmod{3}
	,\\
	\label{Congruence5n}
	\Mspt{2}{5n}\equiv 0 \pmod{5}
	,\\
	\label{Congruence5n1B}
	\Mspt{2}{5n+1}\equiv 0 \pmod{5}
	,\\
	\label{Congruence5n3B}
	\Mspt{2}{5n+3}\equiv 0 \pmod{5}
	.
\end{align}
\end{theorem}

In \cite{GarvanJennings} Garvan and the author proved the congruences
(\ref{Congruence3n1}), (\ref{Congruence5n1A}), and (\ref{Congruence5n3A}) by 
methods quite different from what we use here. There the idea is to give a 
combinatorial refinement of the congruences in terms of an spt-crank. The 
spt-crank is obtained by generalizing the generating function for 
$\Mspt{}{n}$ with an extra variable. By considering this spt-crank at a primitive
third root of unity, one is able to deduce the spt-crank evenly 
divides into three groups the number $\Mspt{}{3n+1}$.
This similarly works with a primitive fifth root of unity for the other two
congruences.

By (\ref{SptEquation1}) and (\ref{SptEquation2}), we see the congruences will 
follow from the congruences between the rank and crank moments.
\begin{theorem}
For $n\ge 0$,
\begin{align}
	N2_2(3n+1) &= M2_2(3n+1)  \pmod{3}
	,\\
	M2_2(5n) &\equiv 0 \pmod{5}
	,\\
	N2_4(5n) + 4N2_2(5n) &\equiv M2_4(5n) + 4M2_2(5n) \pmod{5}
	,\\
	N2_2(5n+1) &\equiv 0\pmod{5}
	,\\
	M2_2(5n+1) &\equiv 0\pmod{5}
	,\\
	N2_4(5n+1) &\equiv M2_4(5n+1) \pmod{5}
	,\\
	N2_2(5n+3) &\equiv M2_2(5n+3) \pmod{5}
	,\\
	N2_4(5n+3) &\equiv M2_4(5n+3) \pmod{5}
.		
\end{align}
\end{theorem}
\begin{proof}
We note that for a fixed prime $p$, $m^p\equiv m\pmod{p}$ for all integers $m$,
and so for $k\ge 1$ each $k$-th moment is congruent modulo $p$ to the 
corresponding $(k+p-1)$-th moment. For example,
\begin{align*}
	N2_4(n) &\equiv N2_2(n) \pmod{3}
	,\\
	M4_6(n) &\equiv M4_2(n) \pmod{5}
.	
\end{align*}

Dividing (\ref{RTwo4ExactRelation}) by $3$ and reducing modulo $3$ yields
\begin{align}\label{RTwo4ExactRelationMod3}
	2nN2_2(n)
	&\equiv
	(2+n)M1_2(n) + 2M2_2(n) + (2+n) M4_2 \pmod{3}
.
\end{align}
Replacing $n$ by $3n+1$ gives
\begin{align*}
	2N2_2(3n+1) &\equiv 2M2_2(3n+1)  \pmod{3}
.
\end{align*}

Reducing (\ref{RTwo4ExactRelation}) modulo $5$ yields
\begin{align}\label{RTwo4ExactRelationMod5}
	N2_4(n) + (4+4n)N2_2(n)
	&\equiv
	(4+n)M1_2(n) + M2_4(n) + (1+4n)M4_2(n) \pmod{5}
.
\end{align}
Reducing (\ref{RTwo6ExactRelation}) modulo $5$ yields
\begin{align}\label{RTwo6ExactRelationMod5}
	(1+2n)N2_4(n) + (4+4n)N2_2(n)	
	&\equiv
	(4+3n+3n^2)M1_2(n) + (4+4n)M2_2(n)+ (1+2n)M2_4(n) 
	\nonumber\\&\quad
	+ (1+2n+2n^2)M4_2(n) \pmod{5}
.
\end{align}
Reducing (\ref{EquationExtraRelation}) modulo $5$ yields
\begin{align}
	\label{EquationExtraRelationMod5}
	(1+n+3n^2)M1_2(n) + (4+2n+3n^2)M2_2(n) + (4+4n+2n^2)M4_2	
	&\equiv 
	0 \pmod{5}	
.
\end{align}

Replacing $n$ by $5n$ in (\ref{RTwo4ExactRelationMod5}), 
(\ref{RTwo6ExactRelationMod5}), and (\ref{EquationExtraRelationMod5}) yields
\begin{align}
	\label{RTwo4ExactRelationMod55n}
	N2_4(5n) + 4N2_2(5n)
	&\equiv
	4M1_2(5n) + M2_4(5n) + M4_2(5n)
	,\\
	\label{RTwo6ExactRelationMod55n}
	N2_4(5n) + 4N2_2(5n)
	&\equiv
	4M1_2(5n) + 4M2_2(5n) + M2_4(5n) + M4_2(5n)
	,\\
	\label{EquationExtraRelationMod55n}
	0
	&\equiv	
	M1_2(5n) + 4M2_2(5n) + 4M4_2(5n)
.
\end{align}
Adding (\ref{RTwo4ExactRelationMod55n}) and (\ref{EquationExtraRelationMod55n}) we
have
\begin{align*}
	N2_4(5n) + 4N2_2(5n) &\equiv M2_4(5n) + 4M2_2(5n) \pmod{5}
.
\end{align*}
Also subtracting (\ref{RTwo6ExactRelationMod55n}) and (\ref{RTwo4ExactRelationMod55n})
we have
\begin{align*}
	M2_2(5n)\equiv 0 \pmod{5}
.
\end{align*}

Replacing $n$ by $5n+1$ in (\ref{RTwo4ExactRelationMod5}), 
(\ref{RTwo6ExactRelationMod5}), and (\ref{EquationExtraRelationMod5}) yields
\begin{align}
	\label{RTwo4ExactRelationMod55n1}
	N2_4(5n+1) + 3N2_2(5n+1)
	&\equiv
	M2_4(5n+1) \pmod{5}
	,\\
	\label{RTwo6ExactRelationMod55n1}
	3N2_4(5n+1) + 3N2_2(5n+1)
	&\equiv
	3M2_2(5n+1) + 3M2_4(5n+1) \pmod{5}
	,\\
	\label{EquationExtraRelationMod55n1}
	0
	&\equiv	
	4M2_2(5n+1)
.
\end{align}
By (\ref{RTwo4ExactRelationMod55n1}), (\ref{RTwo6ExactRelationMod55n1}), and
(\ref{EquationExtraRelationMod55n1}) we find
\begin{align*}
	N2_2(5n+1)\equiv 0\pmod{5},
\end{align*}
which with (\ref{RTwo4ExactRelationMod55n1}) gives
\begin{align*}
	N2_4(5n+1) &\equiv M2_4(5n+1) \pmod{5}
.
\end{align*}

Replacing $n$ by $5n+3$ in (\ref{RTwo4ExactRelationMod5}), 
(\ref{RTwo6ExactRelationMod5}), and (\ref{EquationExtraRelationMod5}) yields
\begin{align}
	\label{RTwo4ExactRelationMod55n3}
	N2_4(5n+3) + N2_2(5n+3)
	&\equiv
	2M1_2(5n+3)+M2_4(5n+3) 
		+ 3M4_2(5n+3) \pmod{5}
	,\\
	\label{RTwo6ExactRelationMod55n3}
	2N2_4(5n+3) + N2_2(5n+3)
	&\equiv
	M2_2(5n+3) + 2M2_4(5n+3) \pmod{5}
	,\\
	\label{EquationExtraRelationMod55n3}
	0
	&\equiv	
	M1_2(5n+3) + 2M2_2(5n+3) + 4M4_2(5n+3)
	\pmod{5}
.
\end{align}
Subtracting (\ref{RTwo6ExactRelationMod55n3}) from 
(\ref{RTwo4ExactRelationMod55n3}) and then applying 
(\ref{EquationExtraRelationMod55n3}) gives
\begin{align*}
	4N2_4(5n+3) \equiv 4M2_4(5n+3) \pmod{5}
.
\end{align*}
This, along with (\ref{RTwo6ExactRelationMod55n3}), gives
\begin{align*}
	N2_2(5n+3) \equiv M2_2(5n+3) \pmod{5}
.
\end{align*}

\end{proof}

Using standard techniques for modular forms, we can establish additional 
congruences for $\Mspt{}{n}$ and $\Mspt{2}{n}$. We define the operators 
$U_m$ and $S_{m,r}$ on formal power series by
\begin{align*}
	U_m \Parans{\sum_{n=0}a(n)q^n}
	&=
	\sum_{n=0}a(mn)q^{n}	
	,\\
	S_{m,r} \Parans{\sum_{n=0}a(n)q^n}
	&=
	\sum_{n=0}a(mn+r)q^{mn+r}
.
\end{align*}
Proposition 2.22 of \cite{Ono1} states $U_m$ 
maps $M_k(\Gamma_0(N),\chi)$ to $M_k(\Gamma_0(\mbox{lcm}(N,m)),\chi)$.
In terms of $S_{m,r}$, we will only need to know where $S_{9,r}$
sends a modular form.
For $(r,3)=1$ each $S_{9,r}$ can be obtained by a linear combination of twists
by modulo $9$ Dirichlet characters. Since each Dirichlet character
modulo $9$ can be written as a product of primitive Dirichlet characters
modulo $1$, $3$, or $9$, by Lemma 4.3.10 of \cite{Miyake} we see 
for $(r,3)=1$ that
$S_{9,r}$ maps $M_k(\Gamma_0(N),\chi)$ to 
$M_k(\Gamma_0(\mbox{lcm}(N,81)),\chi)$.

\begin{theorem}
For $n\ge 0$ we have $\Mspt{2}{9n}\equiv 0\pmod{3}$.
\end{theorem}
\begin{proof}
Since
\begin{align*}
 	\Mspt{2}{n} &=
	\frac{1}{24}\Parans{ M2_4(n) - M2_2(n) - N2_4(n) + N2_2(n) 	}
,
\end{align*}
we work with the rank and crank moments modulo $9$.
Dividing (\ref{RTwo4ExactRelation}) by $3$, replacing $n$ by $9n$, and reducing
modulo $9$ yields
\begin{align*}
	2N2_4(9n) + 7N2_2(9n)
	&\equiv
	M1_2(9n) + M1_4(9n) +5M2_2(9n) 
	+ 6M2_4(9n) 
	\\&\quad
	+ 7M4_2(9n) 
	+ 7M4_4(9n)
	\pmod{9}.
\end{align*}
With this we find that
\begin{align*}
	M2_4(9n) - M2_2(9n) - N2_4(9n) + N2_2(9n)
	&\equiv
	4M1_2(9n) + 4M1_4(9n) + M2_2(9n) + 7M2_4(9n) + M4_2(9n) 
		\\&\quad
		+ M4_4(9n)
	\pmod{9}.
\end{align*}
Noting $3m^2\equiv 3m^4 \pmod{9}$, we can further rearrange terms to get
\begin{align*}
	M2_4(9n) - M2_2(9n) - N2_4(9n) + N2_2(9n)
	&\equiv
	M1_2(9n) + 7M1_4(9n) + M2_2(9n) + 7M2_4(9n) - 2M4_2(9n) 
		\\&\quad
		+ 4M4_4(9n)	
	\pmod{9}
.
\end{align*}
We let $G(q)\in \CommonSpace{2}$ be given by
\begin{align*}
	G(q) &= C1_2(q) + 7C1_4(q) + C2_2(q) + 7C2_4(q) - 2C4_2(q) + 4C4_4(q). 
\end{align*}
This time we use the basis
\begin{align*}
	&\CBrackets{ \frac{\aqprod{-q}{q^2}{\infty}}{\aqprod{q^2}{q^2}{\infty}}  
		\Parans{E_j(q^k)-1}   : j=2,4\mbox{ and } k=1,2,4} 
	\cup 
	\CBrackets{ \frac{\aqprod{-q}{q^2}{\infty}}{\aqprod{q^2}{q^2}{\infty}}  
		\Parans{E_2(q^k)^2-1}  : k=1,2,4}.
\end{align*}
Expressing $G(q)$ in terms of this basis
yields
\begin{align}\label{FunctionForMod9}
	G(q)
	&=
	\frac{\aqprod{-q}{q^2}{\infty}}{\aqprod{q^2}{q^2}{\infty}}
	\frac{1}{240}\left(
	27
	-90E_2(q) - 90E_2(q^2) + 90E_2(q^4)
	+35E_2(q)^2 + 35E_2(q^2)^2 
		\right.\nonumber\\&\qquad\qquad\qquad\qquad\left.	
		- 25E_2(q^4)^2
		+14E_4(q) + 14E_4(q^2) - 10E_4(q^4)
	\right)
.	
\end{align}
From
\begin{align*}
	E_2(q) &\equiv E_2(q) - 27E_2(q^{27}) \pmod{27},
	\\
	9E_2(q) &\equiv 9E_4(q) \pmod{27}, 
\end{align*}
and recalling that $E_2(q) - 27E_2(q^{27})\in M_2(\Gamma_0(27))$, we see 
\begin{align*}
	\frac{1}{240}&\left( 
	27
	-90E_2(q) - 90E_2(q^2) + 90E_2(q^4)
	+35E_2(q)^2 + 35E_2(q^2)^2 - 25E_2(q^4)^2
		\right.\\&\quad\left.
		+14E_4(q) + 14E_4(q^2) 
		- 10E_4(q^4)
	\right)
\end{align*}
is congruent modulo $9$ to some $G_1(q) \in M_{4}(\Gamma_0(108))$. We note 
$G_1(q)$ will have bounded 
rational coefficients with denominators relatively prime to $3$.
Next we have
\begin{align*}
	\Parans{\frac{\aqprod{q}{q}{\infty}^3}{\aqprod{q^3}{q^3}{\infty}}}^3  \equiv 1 \pmod{9}
,
\end{align*}
and so
\begin{align*}
	G(q)
	&\equiv  
	\frac{\aqprod{q^2}{q^2}{\infty}\aqprod{q}{q}{\infty}^8}
	{\aqprod{q^3}{q^3}{\infty}^3\aqprod{q^4}{q^4}{\infty}}
	G_1(q)
	\pmod{9}	.
\end{align*}

We note $S_{9,0}( G(q))\equiv 0\pmod{9}$ if 
$S_{9,0}\Parans{G(q)\aqprod{q^9}{q^9}{\infty}^3} \equiv 0\pmod{9}$,
so we consider
\begin{align*}
	q\frac{\aqprod{q^2}{q^2}{\infty}\aqprod{q}{q}{\infty}^8\aqprod{q^9}{q^9}{\infty}^3}
	{\aqprod{q^3}{q^3}{\infty}^3\aqprod{q^4}{q^4}{\infty}}
	G_1(q)
	&=
	\frac{\eta(2\tau)\eta(\tau)^8\eta(9\tau)^3}{\eta(3\tau)^3\eta(4\tau)}G_1(q)
	.
\end{align*}
By Theorems 1.64 and 1.65 of \cite{Ono1} the latter is an element of
$M_8\Parans{\Gamma_0(1728),\Jac{6}{\cdot}}$.

Since $1728=2^6\cdot 3^3$,
we know
\begin{align*}
	S_{9,1} \Parans{
		\frac{\eta(2\tau)\eta(\tau)^8\eta(9\tau)^3}{\eta(3\tau)^3\eta(4\tau)}G_1(q)
	}
	\in
	M_8\Parans{\Gamma_0(5184),\Jac{6}{\cdot}}.
\end{align*}
The Sturm bound for this space is $6912$, so we verify
\begin{align*}
	S_{9,1}\Parans{
	\frac{\eta(2\tau)\eta(\tau)^8\eta(9\tau)^3}{\eta(3\tau)^3\eta(4\tau)}G_1(q) }
	&\equiv 0 \pmod{9}
\end{align*}
by checking the congruence holds out to $q^{7000}$. This in turn implies
$S_{9,0}(G(q)) \equiv 0\pmod{9}$, which gives the congruence
$\Mspt{2}{9n}\equiv 0\pmod{3}$.

\end{proof}

It would appear, at least for small primes $\ell$, that
\begin{align*}
	\sum_{n=0}^\infty \Mspt{}{\ell n + \beta_\ell}q^n
	\equiv
	\Parans{\frac{\aqprod{q}{q}{\infty}\aqprod{q^4}{q^4}{\infty}}
		{\aqprod{q^2}{q^2}{\infty}}}^{r_\ell}
	H(q)
	\pmod{\ell}
,
\end{align*} 
where $H(q)\in M_{\frac{\ell-r_\ell}{2}+1}(\Gamma_0(4))$ with the $\beta_\ell$
and $r_\ell$ defined by $1\le \beta_\ell < \ell$, 
$8\beta_\ell\equiv 1\pmod{\ell}$, and
$r_\ell = \frac{8\beta_\ell - 1}{\ell}$.
This should be compared with Theorem 6.1 of \cite{Garvan4}, where Garvan 
derives similar congruences for the spt function up to $\ell=37$. Although we 
do not fully investigate this, we do prove the cases when $\ell=3$ and $5$.

\begin{theorem}\label{TheoremSiftCongruences}
We have
\begin{align}
	\label{Congruece3n2}
	\sum_{n=0}^\infty \Mspt{}{3n+2}q^n
	&\equiv
	\Parans{\frac{\aqprod{q}{q}{\infty}\aqprod{q^4}{q^4}{\infty}}
		{\aqprod{q^2}{q^2}{\infty}}}^5
	\pmod{3},
	\\
	\sum_{n=0}^\infty \Mspt{}{5n+2}q^n
	&\equiv
	\Parans{\frac{\aqprod{q}{q}{\infty}\aqprod{q^4}{q^4}{\infty}}
		{\aqprod{q^2}{q^2}{\infty}}}^3
	\Parans{
		E_2(q) + E_2(q^2) + 4E_2(q^4)
	}
	\nonumber\\	\label{Congruece5n2}&\qquad
	\pmod{5}
	.
\end{align}
Here we note that 
$E_2(q) + E_2(q^2) + 4E_2(q^4)
\equiv 
E_2(q) - 2E_2(q^2) - 2E_2(q^2) + 4E_2(q^4)
\pmod{5}
$,
the latter of which we recognize as a modular form with respect to 
$\Gamma_0(4)$.
\end{theorem}
\begin{proof}
Replacing $n$ by $3n+2$ in (\ref{RTwo4ExactRelationMod3}) yields
\begin{align*}
	N2_2(3n+2) \equiv M1_2(3n+2) + 2M2_2(3n+2) + M4_2(3n+2) \pmod{3}
,
\end{align*}
and so
\begin{align*}
	\Mspt{}{3n+2}
	&\equiv
	2M2_2(3n+2) - 2N2_2(3n+2)
	\\
	&\equiv
	M1_2(3n+2)+M2_2(3n+2) + M4_2(3n+2)
	\pmod{3}
.
\end{align*}

We set $G(q) = C1_2(q)+C2_2(q)+C4_2(q)$ so that $G(q)\in\CommonSpace{1}$.  
Expanding $G(q)$ in terms of the basis 
\begin{align*}
	&\CBrackets{
		\frac{\aqprod{-q}{q^2}{\infty}}{\aqprod{q^2}{q^2}{\infty}} 
		\Parans{E_2(q^n)-1}
		: n=1,2,4
	}
\end{align*}
yields
\begin{align*}
	G(q) &=
	\frac{\aqprod{-q}{q^2}{\infty}}{\aqprod{q^2}{q^2}{\infty}}
	\frac{-1}{12}
	\Parans{E_2(q)+E_2(q^2)+E_2(q^4)-3 }
.
\end{align*}
We use that
\begin{align*}
	E_2(q) &\equiv E_2(q) - 9E_2(q^9)  \pmod{9}
	,\\
	3E_2(q) &\equiv 3 \pmod{9},
\end{align*}
and recall that $E_2(q) - 9E_2(q^{9})\in M_2(\Gamma_0(9))$,
to see that $\frac{-1}{12}(E_2(q)+E_2(q^2)+E_2(q^4)-3)$
is congruent modulo $3$ to a modular form of weight $2$ with respect to
$\Gamma_0(36)$. We let $G_1(q)$ be such a modular form.
Next,
\begin{align*}
	\frac{\aqprod{q^2}{q^2}{\infty}}{\aqprod{q}{q}{\infty}\aqprod{q^4}{q^4}{\infty}}
	&\equiv 
	\frac{\aqprod{q}{q}{\infty}^8\aqprod{q^4}{q^4}{\infty}^8}{\aqprod{q^2}{q^2}{\infty}^8}
	\frac{\aqprod{q^6}{q^6}{\infty}^3}{\aqprod{q^3}{q^3}{\infty}^3\aqprod{q^{12}}{q^{12}}{\infty}^3}
	\pmod{3}
,
\end{align*}
so that
\begin{align*}
	S_{3,2}\Parans{
		\frac{\aqprod{q^2}{q^2}{\infty}}
		{\aqprod{q}{q}{\infty}\aqprod{q^4}{q^4}{\infty}}G_1(q)   
	}
	&\equiv
	\frac{\aqprod{q^6}{q^6}{\infty}^3}{\aqprod{q^3}{q^3}{\infty}^3\aqprod{q^{12}}{q^{12}}{\infty}^3}
	S_{3,2}\Parans{
		\frac{\aqprod{q}{q}{\infty}^8\aqprod{q^4}{q^4}{\infty}^8}
		{\aqprod{q^2}{q^2}{\infty}^8}G_1(q)   
	}
	\pmod{3}
.
\end{align*}
It then only remains to show
\begin{align*}
	S_{3,2}\Parans{
		\frac{\aqprod{q}{q}{\infty}^8\aqprod{q^4}{q^4}{\infty}^8}
		{\aqprod{q^2}{q^2}{\infty}^8}G_1(q)   
	}
	&\equiv
	q^2\frac{\aqprod{q^3}{q^3}{\infty}^8\aqprod{q^{12}}{q^{12}}{\infty}^8}
	{\aqprod{q^6}{q^6}{\infty}^8}	
	\pmod{3}
,
\end{align*}
which is equivalent to
\begin{align}\label{Dissection3n2LastPart}
	U_{3}\Parans{
		\frac{\eta(\tau)^8\eta(4\tau)^8}{\eta(2\tau)^8}G_1(q)   
	}
	&\equiv
	\frac{\eta(\tau)^8\eta(4\tau)^8}{\eta(2\tau)^8}
	\pmod{3}
.
\end{align}
By Theorems 1.64 and 1.65 of \cite{Ono1} we know 
$\frac{\eta(\tau)^8\eta(4\tau)^8}{\eta(2\tau)^8}$ to be an element of 
$M_4(\Gamma_0(4))$. Additionally then 
$\frac{\eta(\tau)^8\eta(4\tau)^8}{\eta(2\tau)^8}G_1(q)$
is an element of $M_6(\Gamma_0(36))$. Since $3$ divides $36$ we have
$U_3\Parans{\frac{\eta(\tau)^8\eta(4\tau)^8}{\eta(2\tau)^8}G_1(q)}$
is also an element of $M_6(\Gamma_0(36))$. Noting
$1\equiv -E_2(q)+2E_2(q^2) \pmod{3}$, we may view
(\ref{Dissection3n2LastPart}) as a congruence between elements
of $M_6(\Gamma_0(36))$. The Sturm bound for this space is $48$,
verifying the congruence for this many terms then proves
(\ref{Congruece3n2}).

Replacing $n$ by $5n+2$ in (\ref{RTwo6ExactRelationMod5}) gives
\begin{align*}
	2N2_2(5n+2) 
	&\equiv
	2M1_2(5n+2) + 2M2_2(5n+2) + 3M4_2(5n+2) \pmod{5}
,
\end{align*}
and so
\begin{align*}
	\Mspt{}{5n+2}
	&\equiv
	3M2_2(5n+2) - 3N2_2(5n+2)
	\\
	&\equiv
	2M1_2(5n+2) + 3M4_2(5n+2)
	\pmod{5}. 
\end{align*}
We take $G(q)\in\CommonSpace{1}$ given by $G(q)=2C1_2(q)+3C4_2(q)$. We find
that
\begin{align*}
	G(q)
	&=
	\frac{\aqprod{-q}{q^2}{\infty}}{\aqprod{q^2}{q^2}{\infty}} 
	\frac{-1}{12}
	\Parans{2E_2(q) + 3E_2(q^4) -5}
.
\end{align*}
We note that $\frac{-1}{12}\Parans{2E_2(q) + 3E_2(q^4) -5}$ is congruent
modulo $5$ to an element, call it $G_1(q)$, of $M_2(\Gamma_0(20))$. Along with
\begin{align*}
	\frac{\aqprod{q}{q}{\infty}^5}{\aqprod{q^5}{q^5}{\infty}} &\equiv 1 \pmod{5}
,
\end{align*}
we find that
\begin{align*}
	G(q) 
	&\equiv	
	\frac{\aqprod{q^{10}}{q^{10}}{\infty}^5}
	{\aqprod{q^{5}}{q^{5}}{\infty}^5\aqprod{q^{20}}{q^{20}}{\infty}^5}
	\frac{\aqprod{q}{q}{\infty}^{24}\aqprod{q^4}{q^4}{\infty}^{24}}
	{\aqprod{q^2}{q^2}{\infty}^{24}}
	G_1(q)
	\pmod{5}
.
\end{align*}
It then only remains to show
\begin{align*}
	S_{5,2}\Parans{
	\frac{\aqprod{q}{q}{\infty}^{24}\aqprod{q^4}{q^4}{\infty}^{24}}
	{\aqprod{q^2}{q^2}{\infty}^{24}}
	G_1(q)}
	&\equiv
	q^2\frac{\aqprod{q^{5}}{q^{5}}{\infty}^8\aqprod{q^{20}}{q^{20}}{\infty}^8}
	{\aqprod{q^{10}}{q^{10}}{\infty}^8}
	\Parans{	E_2(q^5) + E_2(q^{10}) + 4E_2(q^{20})}
	\pmod{5}
.
\end{align*}
This is equivalent to
\begin{align}\label{lastCongruenceFor5n2Dissection}
	U_5\Parans{\frac{\eta(\tau)^{24}\eta(2\tau)^{24}}{\eta(2\tau)^{24}}G_1(q) }
	&\equiv
	\frac{\eta(\tau)^{8}\eta(4\tau)^{8}}{\eta(2\tau)^{8}}
	\Parans{	E_2(q) + E_2(q^2) + 4E_2(q^4)}
	\pmod{5}
.
\end{align}
However we have
$U_5\Parans{\frac{\eta(\tau)^{24}\eta(4\tau)^{24}}{\eta(2\tau)^{24}}G_1(q) }$
is an element of $M_{14}(\Gamma_0(20))$ and
\\$\frac{\eta(\tau)^{8}\eta(4\tau)^{8}}{\eta(2\tau)^{8}}
\Parans{	E_2(q) + E_2(q^2) + 4E_2(q^4)}$
is an element of $M_{6}(\Gamma_0(4))$. Since $E_4(q)\equiv 1 \pmod{5}$, we
can view (\ref{lastCongruenceFor5n2Dissection}) as a congruence between
elements of $M_{14}(\Gamma_0(20))$, the Sturm bound for this space is
$42$. Verifying (\ref{lastCongruenceFor5n2Dissection}) holds past this
power of $q$ proves (\ref{Congruece5n2}).

\end{proof}

Since $\frac{\eta(8\tau)\eta(32\tau)}{\eta(16\tau)}
\in M_{\frac{1}{2}}\Parans{\Gamma_0(256),\Jac{2}{\cdot}}$, 
we can use Theorem \ref{TheoremSiftCongruences} to prove the congruences 
\begin{align*}
	\Mspt{}{27n+26} &\equiv 0 \pmod{3},\\
	\Mspt{}{125n+97} &\equiv 0 \pmod{5},\\
	\Mspt{}{125n+122} &\equiv 0 \pmod{5}
\end{align*}
by verifying the congruence holds for so many initial terms. However, these 
congruences are special cases of the much more general congruence,
\begin{align*}
	\Mspt{}{\frac{\ell^{2m}n+1}{8}}
	\equiv 0
	\pmod{\ell^m}
\end{align*}
for prime $\ell\ge 3$, $m\ge 1$, and $\Jac{-n}{\ell}\equiv 1$,
which is Theorem 1.4 of \cite{ABL} by Ahlgren, Bringmann, and Lovejoy.

\bibliographystyle{abbrv}
\bibliography{usingPdeRef}

\end{document}